\documentclass[a4paper
, reqno
]{amsart}

\usepackage{amsmath, amssymb, amsthm}
\usepackage{mathtools}
    \mathtoolsset{showonlyrefs=true}
\usepackage[margin=3.5cm, truedimen]{geometry}
\usepackage{enumerate}

\usepackage{color}
\usepackage{graphicx}
\usepackage[abs]{overpic}
\usepackage[title]{appendix}
\usepackage[all]{xy}
\usepackage{mathrsfs}
\usepackage{comment}

\theoremstyle{definition}
\newtheorem{thm}{Theorem}[section]
\newtheorem{dfn}[thm]{Definition}

\newtheorem{lem}[thm]{Lemma}

\newtheorem{cor}[thm]{Corollary}

\newcommand{\R}{\mathbb{R}}

\newcommand{\Z}{\mathbb{Z}}

\newcommand{\id}{\operatorname{id}}

\newcommand{\ui}{[0,1]}

\numberwithin{equation}{section}

\usepackage{import}
\usepackage{xifthen}
\usepackage{pdfpages}
\usepackage{transparent}

\makeatletter
\renewenvironment{proof}[1][\proofname]{\par
  \pushQED{\qed}%
  \normalfont \topsep6\p@\@plus6\p@\relax
  \trivlist
  \item\relax
  {\bfseries
  #1\@addpunct{.}}\hspace\labelsep\ignorespaces
}{%
  \popQED\endtrivlist\@endpefalse
}
\makeatother

\newcommand{\card}[1]{|#1|}
\newcommand{\concsp}[1]{\mathscr{C}(#1)}

\newcommand{\crit}[2]{C_{#1}(#2)}

\newcommand{\phimap}[2]{\varphi_{#1}([#2])}
\newcommand{\Phimap}[1]{\Phi([#1])}

\newcommand{\floor}[1]{\left\lfloor #1 \right\rfloor}

\usepackage{subfiles}

\title{Concordance of Morse functions on manifolds}
\author{Ryosuke Ota}
\address{Joint Graduate School of Mathematics for Innovation, Kyushu University, 744 Motooka, Nishi-ku, Fukuoka 819-0395, Japan}
\email{ota.ryosuke.509@s.kyushu-u.ac.jp}

\begin{document}

\begin{abstract}
    In this paper, the concordance of Morse functions is defined, and a necessary and sufficient condition for given two Morse functions to be concordant is presented and is compared with the cobordism criterion. 
    Cobordism of Morse functions on smooth closed manifolds is an equivalence relation defined by using cobordisms of manifolds and fold maps. 
    Given two Morse functions, it is important to decide whether they are cobordant or not, and this problem was first solved for surfaces and then for manifolds of general dimensions by Ikegami–Saeki, Kalmár, and Ikegami. 
    On the other hand, for Morse functions on the same manifold, we can consider a stronger equivalence relation called concordance. 
\end{abstract}

\maketitle

\section{Introduction}

The oriented cobordism group of Morse functions is introduced by Ikegami--Saeki~\cite{Ikegami--Saeki}. 
They computed it for orientable surfaces. 
Afterward, Kalm\'ar~\cite{kalmar2005cobordism} studied for unorientable surfaces and  Ikegami~\cite{Ikegami} determined the $n$-dimensional oriented and non-oriented cobordism group for general $n$. 
In the present paper, we introduce the concordance of Morse functions and determine the concordance set. 
In particular, we can determine whether two Morse functions are concordant or not by similar invariants used in the cobordism case. 
Hence, we can compare two equivalence relations, concordance and cobordism through those invariants. 

The cobordism theory of smooth maps is begun by Thom~\cite{thom}. 
He proved that the cobordism group of embeddings is isomorphic to a homotopy group of some space by using the Pontrjagin--Thom construction. 
Then Wells~\cite{wells} used a similar method to study the cobordism group of immersions. 
Rim\'nyi--Sz\H{u}cs~\cite{Rimanyi--Szucs} generalized these results to the cobordism group of maps with given singularities. 
The three above dealt with the cases of non-negative codimensions and cobordism groups were determined by using the method of algebraic topology. 
On the other hand,  Ikegami--Saeki, Kalm\'ar, and Ikegami studied the cobordism group of Morse functions, which is the case of negative codimension, by using geometric methods: Stein factorizations (Reeb graphs) and the elimination of cusps by Levine~\cite{levine1965elimination}. 

The concordance originally appeared in knot theory and is a vital equivalence relation. 
We introduce a similar notion to Morse functions. 
This is the first example where a concordance-like equivalence relation is introduced to smooth maps of negative codimensions. 
To be more specific, for a given closed connected $n$-dimensional manifold $M$, the concordance is an equivalence relation on the set of the Morse functions on $M$. 
Let $\concsp{M}$ be the corresponding quotient set. 
Then the main theorem of the paper is the following, which determines the set $\concsp{M}$. 
\begin{thm}
    Let $M$ be a closed connected $n$-dimensional manifold. 
    \begin{enumerate}
        \item 
        When $n$ is even, the set $\concsp{M}$ is in one-to-one correspondence with $\Z^{\floor{n/2}}$. 
        \item 
        When $n$ is odd, the set $\concsp{M}$ is in one-to-one correspondence with $\Z^{\floor{n/2}}\oplus \Z_2$.
    \end{enumerate}
\end{thm}

The paper is organized as follows. 
In \S 2, we recall the cobordism of Morse functions and introduce the concordance of them. We also outline the method of eliminating cusps by Levine. 
It is used for concordance as well as cobordism. 
In \S 3 we prove the main theorem of the paper and compare it with the result for cobordism. 

Throughout the paper, all manifolds and maps between them are smooth of class $C^\infty$. 
For a smooth map $f$, the singular point set of $f$ is denoted by $S(f)$, and for a topological space $X$, the number of the connected components of $X$ is denoted by $\card{X}$. 

The author extends his sincere thanks to Professor Osamu Saeki for his guidance and support. 
Ryosuke Ota is supported by WISE program (MEXT) at Kyushu University. 
This work was supported by the Grant-in-Aid for Scientific Research (S), JSPS KAKENHI Grant Number 23H05437.

\medskip

\section{Preliminaries}

Let $f: M\to \R$ be a smooth real-valued function on a manifold $M$. 
The function $f$ is called a \textit{Morse function} if all critical points of them are non-degenerate, that is, $f$ has non-singular Hessian at each critical point. 

We recall two types of singularities of smooth maps into the plane. 

\begin{dfn}
    Let $F: X\to\R^2$ be a smooth map of an $m$-dimensional manifold $X$ of dimension $m\ge 2$. 
    A singular point $p\in X$ of $F$ is said to be a \textit{fold point} if there is a coordinates $(u, x_1, \dots, x_{m-1})$ centered at $p$ and $(y_1, y_2)$ at $F(p)$ such that $F$ is locally represented by 
    \begin{equation}
        \left\{ \,
        \begin{aligned}
            y_1\circ F &= u, \\
            y_2\circ F &= -x_1^2 - \dots -x_\lambda^2 + x_{\lambda+1}^2 + \dots + x_{m-1}^2
        \end{aligned}
        \right.
    \end{equation}
    for some $0\le \lambda \le m-1$. 
    We call $\max\{\lambda, m-1-\lambda\}$ the \textit{absolute index} of the fold point $p$. 

     We call $p$ a \textit{cusp point} if there is a coordinates $(u_1, u_2, x_1, \dots, x_{m-2})$ centered at $p$ and $(y_1, y_2)$ at $F(p)$ such that $F$ is locally represented by 
    \begin{equation}
        \left\{ \,
        \begin{aligned}
            y_1\circ F &= u_1, \\
            y_2\circ F &= u_1 + u_1u_2^3 -x_1^2 - \dots -x_\lambda^2 + x_{\lambda+1}^2 + \dots + x_{m-2}^2
        \end{aligned}
        \right.
    \end{equation}
    for some $0\le \lambda \le m-2$. 
    We call $\max\{\lambda, m-2-\lambda\}$ the \textit{absolute index} of the cusp point $p$. 

    The absolute index of a fold point or cusp point $p$ is independent of coordinates of $p$ and $F(p)$. 

\end{dfn}

We say that $F$ is \textit{generic} if $F$ has only fold points and cusp points as its singular points. 
In particular, $F$ is a \textit{fold map} if $F$ has only fold points as its singular points. 
If $F$ is generic, the singular set $S(F)$ is a 1-dimensional submanifold of $X$, and particularly, the set of the cusp points is a discrete subset. 
It is well known that every smooth map of $X$ into the plane can be approximated by a generic map. 
By the local form at a fold point, the absolute index of fold points is constant on each component of $S(F)\setminus \{\text{the cusp points}\}$. Hence indices can be defined for each component of it. 
The absolute indices adjacent to a cusp point of absolute index $i$ are as in Fig.~\ref{fig: non minimum absolute index of cusp} (resp. Fig.~\ref{fig: minimum absolute index of cusp}) if $i> (n-1)/2$ (resp. $i=(n-1)/2$). 
Both figures show the image of a neighborhood of a cusp point by $F$ and each integer represents the absolute index of a fold point or a cusp point. 

\begin{figure}[b]
    \centering
    \includegraphics[width = 0.25\columnwidth]{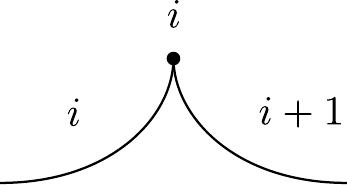}
    \caption{absolute indices of a cusp point and adjacent fold points ($i> (n-1)/2$)}
    \label{fig: non minimum absolute index of cusp}
\end{figure}

\begin{figure}[t]
    \centering
    \includegraphics[width = 0.25\columnwidth]{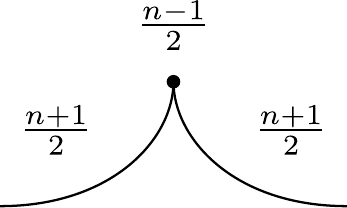}
    \caption{absolute indices of a cusp point and adjacent fold points ($i = (n-1)/2$)}
    \label{fig: minimum absolute index of cusp}
\end{figure}

If $X$ is connected and $m$ is greater than 2, there is a method to eliminate some pairs of cusp points from a given generic map $F$ \cite{levine1965elimination}. 
Here, we outline the method. 
Let $p_1$ and $p_2$ be two cusp points of $F$. 
If $\{p_1, p_2\}$ is an appropriate pair called ``matching pair", we can eliminate two cusp points $p_1$ and $p_2$ by a homotopy of $F$ supported in a small neighborhood of $\lambda (\ui)$, where $\lambda: \ui\to X$ is some arc such that $\lambda (0) = p_1$, $\lambda (1) = p_2$, and $\lambda (\ui)\cap S(F)$. 

We recall the notion of oriented and unoriented cobordism of Morse functions \cite{Ikegami}. 

\begin{dfn}
    Let $M_0$ and $M_1$ be two closed oriented manifolds of the same dimension $n$. 
    Two Morse function $f_0: M_0\to\R$ and $f_1: M_1\to\R$ are said to be \textit{oriented cobordant} if there is a $(n+1)$-dimensional oriented manifold $X$ and a fold map $F: X\to \R\times\ui$ such that
    \begin{enumerate}
        \item 
        $X$ is a oriented cobordism between $M_0$ and $M_1$, that is, $\partial X = M_0\sqcup (-M_1)$, where $-M_1$ is the manifold $M_1$ with the opposite orientation, and
        
        \item 
        for some $\varepsilon >0$, we have
        \begin{align}
            F|_{M_0\times [0,\varepsilon)} &= f_0\times \id_{[0, \varepsilon)}, ~~\text{and}\\
            F|_{M_1\times (1-\varepsilon, 1]} &= f_1\times \id_{(1-\varepsilon, 1]}, 
        \end{align}
        where collar neighborhoods of $M_0$ and $M_1$ in $X$ are identified with $M_0\times [0,\varepsilon)$ and $M_1\times (1-\varepsilon, 1]$, respectively. 
    \end{enumerate}
\end{dfn}
By ignoring the orientations, we can obtain the definition of \textit{unoriented cobordism} of Morse functions. 
The oriented cobordism is an equivalence relation on the set of Morse functions on oriented $n$-manifolds. Then the quotient set has a group structure and is called the \textit{$n$-dimensional oriented cobordism group of Morse functions}, where the addition is defined by the disjoint union. 
In a similar way, we can define the \textit{$n$-dimensional unoriented cobordism group of Morse functions}. 

We introduce concordance, which is another equivalence relation of Morse functions. 

\begin{dfn}
    Let $f_0$, $f_1: M\to \R$ be two Morse functions on the same manifold $M$. 
    Then $f_0$ and $f_1$ is said to be \textit{concordant} if there is a fold map $F: M\times\ui\to \R\times\ui$ such that for some $\varepsilon >0$, 
    \begin{align}
        F|_{M_0\times [0,\varepsilon)} &= f_0\times \id_{[0, \varepsilon)}, ~~\text{and}\\
        F|_{M_1\times (1-\varepsilon, 1]} &= f_1\times \id_{(1-\varepsilon, 1]}.
    \end{align}
\end{dfn}

Given a manifold $M$, the concordance is an equivalence relation on the set of the Morse functions on $M$. 
Let $\concsp{M}$ denote the corresponding quotient set. 
Replacing the cylinder $M\times\ui$ by a cobordism between $M$ and itself in the definition above, we can obtain the notion of cobordism of Morse functions on the same manifold \cite{Ikegami--Saeki}. 
Concordance is a stronger equivalence relation than cobordism.


We mention the cobordism invariants used in the result by Ikegami~\cite{Ikegami} about the classification of Morse functions by cobordism. 
For a Morse function $f$ on a $n$-dimensional closed manifold and an integer $0\le \lambda\le n$, Let $\crit{\lambda}{f}$ be the number of the critical points of index $\lambda$ of $f$. 
The value $\phimap{\lambda}{f}\in \Z$ below (hence $\Phimap{f}\in \Z^{\floor{n/2}}$) is well defined \cite[Lemma 4.1]{Ikegami}, where $[f]$ denotes the cobordism class of $f$:
\begin{align}
    \phimap{\lambda}{f} 
    &= \crit{\lambda}{f} - \crit{n-\lambda}{f},
    \label{eq: definition of phi_lambda in preliminaries}\\
    \Phimap{f} 
    &= 
    \left(
    \phimap{\floor{(n+3)/2}}{f}, \phimap{\floor{(n+3)/2} + 1}{f}, \dots, \phimap{n}{f}
    \right). 
    \label{eq: definition of Phi in preliminaries}
\end{align}
We can use these invariants for concordance since concordance is a stronger equivalence relation than cobordism. 
For a Morse function $f$ on a closed manifold $M$ of dimension $n=2k+1$, we define $\sigma (f) \in \Z_2$ by 
\begin{equation}\label{eq: definition of sigma}
    \sigma(f) \coloneqq \sum_{\lambda=0}^{k} \crit{\lambda}{f} \pmod{2}. 
\end{equation}
By Lemma~\ref{lem: well-definedness of bar sigma}, we obtain another invariant $\bar\sigma : \concsp{M}\to\Z_2$ for concordance. 
Then, we can state the main result more precisely. 
\begin{thm}\label{thm: main result}
    Let $M$ be a closed connected $n$-dimensional manifold. 
    \begin{enumerate}
        \item 
        When $n$ is even, the map $\Phi: \concsp{M}\to \Z^{\floor{n/2}}$ is bijective. 

        \item 
        When $n$ is odd, the map $\Phi\oplus \bar\sigma: \concsp{M}\to \Z^{\floor{n/2}}\oplus\Z_2$ is bijective. 
    \end{enumerate}
\end{thm}

\medskip

\section{Proof of the main result}

We start by proving that $\bar\sigma$ is well defined. 

\begin{lem}\label{lem: well-definedness of bar sigma}
    Let $f_0$ and $f_1$ be two Morse functions on a closed manifold $M$ of dimension $n = 2k+1$. 
    Then $\sigma (f_0) = \sigma (f_1)$ holds.  
\end{lem}

\begin{proof}
    Take a concordance $F: M\times\ui\to\R\times\ui$ between $f_0$ and $f_1$. 
    Perturbing $F$, we may assume that $F_2 \coloneqq \operatorname{pr}_2\circ F: M\times\ui\to\ui$ is a Morse function, where $\operatorname{pr}_2: \R\times\ui\to\ui$ is the standard projection to the second component.  
    Since each component $c$ of the 1-dimensional manifold $S(F)$ is an arc or a circle, $c$ is one of the four types below (see Fig.~\ref{fig: four types of components of S(F)}). 
    \begin{enumerate}
        \item 
        The component $c$ is an arc connecting two critical points $p_1$ and $p_2$ of $f_0$. 
        Then $F_2$ has an odd number of critical points on $c$. 
        In addition, the indices of $p_1$ and $p_2$ as critical points of $f_0$ are $\lambda$ and $2k+1-\lambda$, respectively. 

        \item 
        The component $c$ is an arc connecting two critical points $q_1$ and $q_2$ of $f_1$. 
        Then $F_2$ has an odd number of critical points on $c$. 
        In addition, the indices of $q_1$ and $q_2$ as critical points of $f_1$ are $\lambda$ and $2k+1-\lambda$, respectively. 

        \item 
        The component $c$ is an arc connecting two critical points $p_3$ of $f_0$ and $q_3$ of $f_1$. 
        Then $F_2$ has an even number of critical points on $c$. 
        In addition, the indices of $p_3$ and $q_3$ are the same as a critical point of $f_0$ and $f_1$ respectively. 
        Let $i(c)$ denote the index. 

        \item 
        The component $c$ is a circle. 
        Then $F_2$ has an even number of critical points on $c$. 
    \end{enumerate}

    \begin{figure}[b]
        \centering
        \includegraphics[width=0.5\columnwidth]{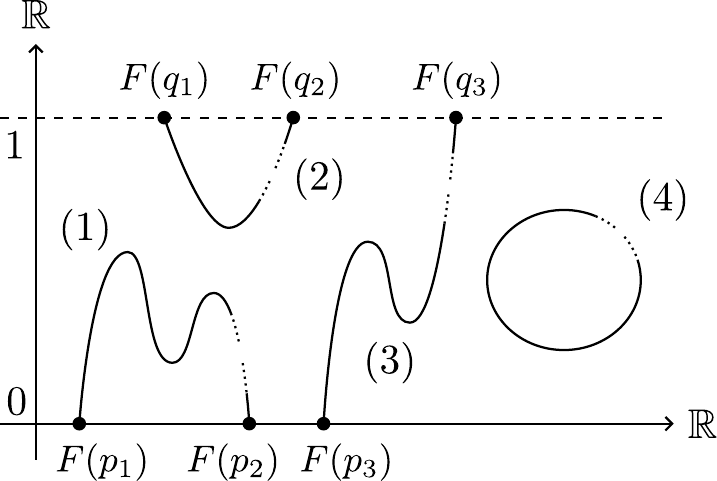}
        \caption{The image by $F$ of four types of $c$}
        \label{fig: four types of components of S(F)}
    \end{figure}

    Since $2k+1-\lambda\ge k+1 > k$ for $0\le \lambda\le k$, the sum $\sum_{\lambda =0}^k\crit{\lambda}{f_0}$ is equal to the number of the components of type (1) and type (3) with $i(c) \le k$. 
    Hence we can calculate $\sigma (f_0)$ in the following way. 
    \begin{align}
        \sigma (f_0)
        &\equiv
        \sum_{\lambda =0}^k \crit{\lambda}{f_0} \pmod{2}\\
        &=
        \card{\{c\mid \text{$c$ is of type (1)}\}} + \card{\{c\mid \text{$c$ is of type (3), }i(c)\le k\}}
        \label{eq: calculation of sigma (f_0)}\\
        &\equiv
        \sum_{\text{$c$ is of type (1)}} \card{ S(F_2)\cap c} + \card{\{c\mid \text{$c$ is of type (3), }i(c)\le k\}}\pmod{2}
    \end{align}
    We can also calculate $\sigma (f_1)$ in a similar way. 
    \begin{align}
        \sigma (f_1) 
        &\equiv
        \card{\{c\mid \text{$c$ is of type (2)}\}} + \card{\{c\mid \text{$c$ is of type (3), }i(c)\le k\}} \pmod{2}
        \label{eq: calculation of sigma (f_1)}\\
        &\equiv
        \sum_{\text{$c$ : type (2)}} \card{ S(F_2)\cap c} + \card{\{c\mid \text{$c$ : type (3), }i(c)\le k\}}\pmod{2}. 
    \end{align}
    In addition, Since $F_2$  has an even number of critical points on the type (3) and (4) components, we obtain
    \begin{align}
        \sigma (f_0) - \sigma (f_1)
        &=
        \sigma (f_0) + \sigma (f_1) \\
        &\equiv
        \sum_{\text{$c$ : type (1)}} \card{ (S(F_2)\cap c} + \sum_{\text{$c$ : type (2)}} \card{S(F_2)\cap c} \pmod{2} \\
        &\equiv
        \card{S(F_2)} \pmod{2}. 
    \end{align}
    Therefore, what we have to prove is 
    \begin{equation}\label{eq: |S(F_2)| is even for any Morse function F_2}
    \card{S(F_2)} \equiv 0 \pmod{2}. 
    \end{equation}
    Since $F_2$ is a Morse function, we have $\chi (M\times \ui, M\times \{0\}) = \sum_{\lambda =0}^{2k+2} (-1)^\lambda\crit{\lambda}{F_2}$. 
    Hence, 
    \begin{align}
        \card{S(F_2)}
        &=
        \sum_{\lambda =0}^{2k+2} \crit{\lambda}{F_2} \\
        &\equiv
        \sum_{\lambda =0}^{2k+2} (-1)^\lambda\crit{\lambda}{F_2} \pmod{2} \\
        &=
        \chi (M\times \ui, M\times \{0\}) \\
        &=
        \chi (M\times \ui) - \chi (M) \\
        &=
        0.
    \end{align}
    We have proven \eqref{eq: |S(F_2)| is even for any Morse function F_2} and it means $\sigma (f_0) = \sigma (f_1) = 0$.
\end{proof}

We prove the main theorem for $n\ge 2$. 
We can prove it for $n=1$ in a different but more direct way. 

\begin{proof}[Proof of Theorem~\ref{thm: main result}(1)]

Since we can prove $\Phi$ is surjective in the same way in \cite{Ikegami}, we consider only the injectivity. 

Let $f_0$ and $f_1$ be two Morse functions on $M$. 
Suppose that they satisfy $\Phimap{f_0} = \Phimap{f_1}$. 
First, take a generic map $F : M\times\ui\to\R\times\ui$ satisfying $F|_{M_0\times [0,\varepsilon)} = f_0\times \id_{[0, \varepsilon)}$ and $F|_{M_1\times (1-\varepsilon, 1]} = f_1\times \id_{(1-\varepsilon, 1]}$ for some $0<\varepsilon<1/2$. 
If $F$ has no cusp point, $F$ is a fold map and it means that $f_0$ and $f_1$ are concordant. 
Below we assume that $F$ has a cusp point. 

Let $\mu$ be the maximum value of the absolute indices of fold points adjacent to a cusp point. 
If $\mu-1 > n-\mu$ holds, we can eliminate all the cusp points of absolute index $\mu-1$ \cite[\S 4]{levine1965elimination}. 
Repeating this procedure, we obtain a generic map $F$ which has no cusp point of index greater than $(n-1)/2$. 
Since each absolute index of a cusp point is greater than or equal to $n/2$ when $n$ is even, $F$ has no cusp when $n$ is even. Therefore $F$ is a concordance between $f_0$ and $f_1$ and we have completed the proof. 

\end{proof}

\begin{proof}[Proof of Theorem~\ref{thm: main result}(2)]

First, we prove that $\Phi\oplus \bar\sigma$ is surjective. 

Let $n=2k+1$. 
Suppose that an arbitrary element $(a_{k+2}, a_{k+3}\dots, a_{2k+1}, \ell)$ in $\Z^k\oplus \Z_2$ is given. 
There is a Morse function $h_0$ on $M$ such that $\Phimap{h_0} = (a_{k+2}, a_{k+3}\dots, a_{2k+1})$ as well as $n$ is even. 
Since we can adopt $h_0$ if $\sigma (h_0) = \ell$, consider $\sigma (h_0)\neq \ell$. 
In this case, we create a pair of critical points of indices $k$ and $k+1$. 
Then the resulting Morse function $h$ satisfies the following. 
\begin{align}
    \Phimap{h} &= \Phimap{h_0} = (a_{k+2}, a_{k+3}\dots, a_{2k+1}), \\
    \sigma (h) 
    &\equiv
    \sum_{\lambda = 0}^k \crit{\lambda}{h} \pmod{2}\\
    &=
    \sum_{\lambda = 0}^{k-1} \crit{\lambda}{h_0} + (\crit{k}{h_0} + 1) \\
    &\equiv
    \sigma (h_0) +1 \pmod{2}\\
    &\equiv
    \ell \pmod{2}.
\end{align}
This completes the proof of surjective part. 

Next, we check the injectivity. 

Given two Morse functions $f_0$ and $f_1$ on $M$ satisfying $\Phimap{f_0} = \Phimap{f_1}$ and $\sigma (f_0) = \sigma (f_1)$, we need to prove that $f_0$ and $f_1$ are concordant. 
We can take a generic map $F: M\times\ui\to\R\times\ui$ which satisfies $F|_{M\times [0,\varepsilon)} = f_0\times \id_{[0, \varepsilon)}$ and $F|_{M\times (1-\varepsilon, 1]} = f_1\times \id_{(1-\varepsilon, 1]}$ for some $0<\varepsilon<1/2$ and has no cusp point of absolute index greater than $(n-1)/2 = k$. 
It implies that every cusp point of $F$ has the absolute index $k$. 
Since every pair of cusp points of absolute index $k$ is a matching pair (\cite{levine1965elimination}), we can eliminate all the cusp points when $F$ has an even number of cusps. 
In the following, we prove that $F$ has an even number of cusps under the assumption that $\sigma (f_0) = \sigma (f_1)$. 
By perturbing $F$ after composing a self-diffeomorphism of $\R\times\ui$, we may assume that $F_2: M\times\ui\to\ui$ is a Morse function and the image of a neighborhood of the cusp points in $S(F)$ is as shown in Fig.~\ref{fig: image of neighborhood of the cusps of F}. 
These two conditions are compatible since each cusp point is always a regular point of $F_2$ whenever the latter condition holds. 
Then, each component $c'$ of $S(F) \setminus \{\text{the cusp points}\}$ is one of the seven types below (see Fig.~\ref{fig: seven types of components of S(F) setminus cusps}).
\begin{enumerate}[(1)$'$]
    \item 
    The component $c'$ is an arc connecting two critical points $p_1$ and $p_2$ of $f_0$.  
    Then $F_2$ has an odd number of critical points on $c'$. 
    In addition, the indices of $p_1$ and $p_2$ as critical points of $f_0$ are $\lambda$ and $2k+1 - \lambda$, respectively. 

    \item 
    The component $c'$ is an arc connecting two critical points $q_1$ and $q_2$ of $f_1$.  
    Then $F_2$ has an odd number of critical points on $c'$. 
    In addition, the indices of $q_1$ and $q_2$ as critical points of $f_1$ are $\lambda$ and $2k+1 - \lambda$, respectively. 
    
    \item
    The component $c'$ is an arc connecting two critical points $p_3$ of $f_0$ and $q_3$ of $f_1$. 
    Then $F_2$ has an even number of critical points on $c'$. 
    In addition, the indices of $p_3$ and $q_3$ are the same as a critical point of $f_0$ and $f_1$ respectively. 
    
    \item
    The component $c'$ is a circle. 
    Then $F_2$ has an even number of critical points on $c'$. 

    \item 
    The component $c'$ is an arc connecting a critical point $p_4$ of $f_0$ and a cusp point $x_1$ of $F$. 
    Then $F_2$ has an even number of critical points on $c'$. 

    \item 
    The component $c'$ is an arc connecting two cusp points $x_1$ and $x_2$ of $F$. 
    Then $F_2$ has an odd number of critical points on $c'$. 

    \item 
    The component $c'$ is an arc connecting a critical point $q_4$ of $f_1$ and a cusp point $x_2$ of $F$. 
    Then $F_2$ has an odd number of critical points on $c'$. 
\end{enumerate}
In summary, the parity of the number of connected components of $S(F)\cap c'$ is the following. 
\begin{equation}\label{eq: the number of critical points of F_2 on c'}
    \card{S(F_2)\cap c'}
    \equiv
    \begin{cases}
    1 \pmod{2} & (\text{$c'$ is of type (1)$'$}) \\
    1 \pmod{2} & (\text{$c'$ is of type (2)$'$}) \\
    0 \pmod{2} & (\text{$c'$ is of type (3)$'$}) \\
    0 \pmod{2} & (\text{$c'$ is of type (4)$'$}) \\
    0 \pmod{2} & (\text{$c'$ is of type (5)$'$}) \\
    1 \pmod{2} & (\text{$c'$ is of type (6)$'$}) \\
    1 \pmod{2} & (\text{$c'$ is of type (7)$'$}) 
    \end{cases}. 
\end{equation}

\begin{figure}[btp]
    \centering
    \includegraphics[width = 0.45\columnwidth]{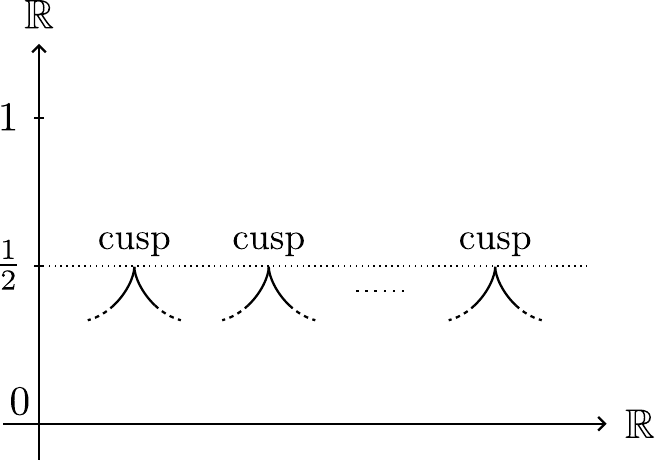}
    \caption{The image by $F$ of a neighborhood of the cusp points in $S(F)$}
    \label{fig: image of neighborhood of the cusps of F}
\end{figure}

\begin{figure}[btp]
    \centering
    \includegraphics[width = 0.8\columnwidth]{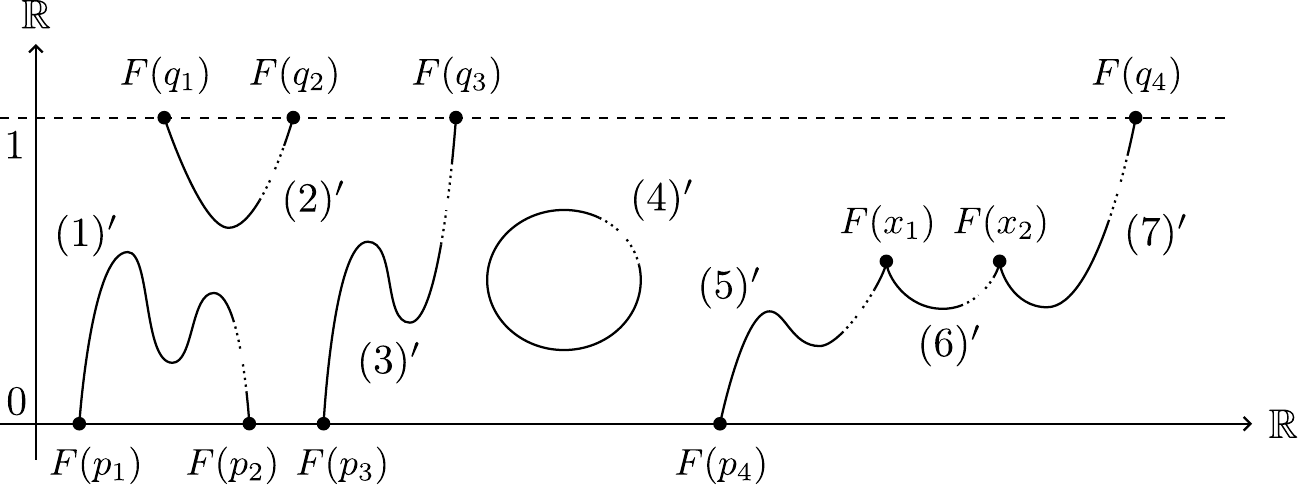}
    \caption{The image by $F$ of seven types of $c'$}
    \label{fig: seven types of components of S(F) setminus cusps}
\end{figure}

Let $c$ be a component of $S(F)$. 
we calculate the number of the cusp points of $F$ on $c$ modulo 2, that we call $l(c)$. 
The component $c$ is one of the four types below, as in the proof of Lemma~\ref{lem: well-definedness of bar sigma}. 
\begin{enumerate}
    \item 
    The component $c$ is an arc connecting two critical points $p_1$ and $p_2$ of $f_0$. 
    Then $F_2$ has an odd number of critical points on $c$. 
    In addition, the indices of $p_1$ and $p_2$ as critical points of $f_0$ are $\lambda$ and $2k+1-\lambda$, respectively. 

    \item 
    The component $c$ is an arc connecting two critical points $q_1$ and $q_2$ of $f_1$. 
    Then $F_2$ has an odd number of critical points on $c$. 
    In addition, the indices of $q_1$ and $q_2$ as critical points of $f_1$ are $\lambda$ and $2k+1-\lambda$, respectively. 

    \item 
    The component $c$ is an arc connecting two critical points $p_3$ of $f_0$ and $q_3$ of $f_1$. 
    Then $F_2$ has an even number of critical points on $c$. 
    In addition, the indices of $p_3$ and $q_3$ are the same as a critical point of $f_0$ and $f_1$ respectively. 
    Let $i(c)$ denote the index. 

    \item 
    The component $c$ is a circle. 
    Then $F_2$ has an even number of critical points on $c$. 
\end{enumerate}
We calculate the parity of $l(c)$ for each of $c$ types 1 to 4. For every case, $c$ is a union of $c'$'s. 
\begin{enumerate}
    \item 
    When $l(c)>0$, $c$ is the union of two $c'$ of type (5)$'$ and $(l(c)-1)$ $c'$ of type (6)$'$. 
    Hence, by using \ref{eq: the number of critical points of F_2 on c'} we obtain
    \begin{align}
        \card{S(F_2)\cap c} \equiv 2\cdot 0 + (l(c)-1)\cdot 1 \equiv l(c)+1 \pmod{2}. 
    \end{align}
    It also holds when $l(c)=0$. 

    \item 
    When $l(c)>0$, we obtain 
    \begin{align}
        \card{S(F_2)\cap c} \equiv 2\cdot 1 + (l(c)-1)\cdot 1 \equiv l(c)+1 \pmod{2}. 
    \end{align}
    It also holds when $l(c)=0$. 

    \item 
    When $l(c)>0$, we obtain 
    \begin{align}
        \card{S(F_2)\cap c} \equiv 1\cdot 0 + (l(c)-1)\cdot 1\cdot 1 \equiv l(c) \pmod{2}. 
    \end{align}
    It also holds when $l(c)=0$. 

    \item 
    When $l(c)>0$, we obtain 
    \begin{align}
        \card{S(F_2)\cap c} \equiv l(c)\cdot 1 \equiv l(c) \pmod{2}. 
    \end{align}
    It also holds when $l(c)=0$. 
\end{enumerate}
From the above, we can calculate the number of the cusps of $F$. 
\begin{align}
    \sum_{\text{$c$ : component of $S(F)$}} l(c) &=
    \sum_{\text{$c$ : type(1)}} l(c)
    + \sum_{\text{$c$ : type(2)}} l(c)
    + \sum_{\text{$c$ : type(3)}} l(c)
    + \sum_{\text{$c$ : type(4)}} l(c) \\
    &\equiv
    \sum_{\text{$c$ : type(1)}} (\card{S(F_2)\cap c} + 1)
    + \sum_{\text{$c$ : type(2)}} (\card{S(F_2)\cap c} + 1) \\
    &~~~~~
    + \sum_{\text{$c$ : type(3)}} \card{S(F_2)\cap c}
    + \sum_{\text{$c$ : type(4)}} \card{S(F_2)\cap c} \pmod{2}\\
    &=
    \card{S(F_2)} + \card{\{c\mid \text{$c$ is of type (1)}\}} + \card{\{c\mid \text{$c$ is of type (2)}\}} \\
    &\equiv
    \card{\{c\mid \text{$c$ is of type (1)}\}} + \card{\{c\mid \text{$c$ is of type (2)}\}} \pmod{2}
    \label{eq: parity of the number of cusps of F}, 
\end{align}
where we applied \eqref{eq: |S(F_2)| is even for any Morse function F_2}. 
On the other hand, since \eqref{eq: calculation of sigma (f_0)} and \eqref{eq: calculation of sigma (f_1)} also hold when $F$ has cusp points, we obtain
\begin{align}
    0 
    &=
    \sigma (f_0) - \sigma (f_1) \\
    &=
    \sigma (f_0) + \sigma (f_1) \\
    &\equiv
    \card{\{c\mid \text{$c$ is of type (1)}\}} + \card{\{c\mid \text{$c$ is of type (2)}\}} \pmod{2}. \label{eq: the number of c of type (1) or type (3)} 
\end{align}
From \eqref{eq: parity of the number of cusps of F} and \eqref{eq: the number of c of type (1) or type (3)}, we see that $F$ has an even number of cusps. 
Hence, we can eliminate all the cusps of $F$ and the resulting generic map is a fold map. 
It implies that $f_0$ and $f_1$ are concordant and we complete the proof. 
\end{proof}

Theorem~\ref{thm: main result} determines $\concsp{M}$. 
We compare the theorem with those for cobordism. 
The cobordism group of Morse functions was determined in Theorem~2.7, 2.8, and 2.9 in Ikegami~\cite{Ikegami}.
Comparing Theorem~\ref{thm: main result} with these results, we obtain the following corollaries, which show that the concordance is truly a stronger equivalence relation than the cobordism in some dimensions. 
\begin{cor}
    Let $M$ be a closed connected $n$-dimensional manifold. 
    \begin{enumerate}
        \item 
        When $n$ is even, any two unoriented cobordant Morse functions $f_0$ and $f_1$ on $M$ are concordant. 
        \item 
        When $n$ is odd, there is a pair of Morse functions $f_0$ and $f_1$ on $M$ such that they are unoriented cobordant but not concordant. 
    \end{enumerate}
\end{cor}

\begin{cor}
    Let $M$ be a closed connected $n$-dimensional manifold. 
    \begin{enumerate}
        \item 
        When $n\not\equiv 3 \pmod{4}$, any two oriented cobordant Morse functions $f_0$ and $f_1$ on $M$ are concordant. 
        \item 
        When $n\equiv 3 \pmod{4}$, there is a pair of Morse functions $f_0$ and $f_1$ on $M$ such that they are oriented cobordant but not concordant. 
    \end{enumerate}
\end{cor}

\end{document}